\documentclass{amsart}
\usepackage{microtype,amssymb,algorithm,algpseudocode,graphicx,color,enumerate,cite,amsthm}

\def\RR{\ensuremath{\mathbb{R}}}
\newcommand{\CC}{{\mathbb C}}

\newcommand{\V}{{\mathcal V}}
\newcommand{\N}{{\mathcal N}}
\newcommand{\M}{{\mathcal M}}

\newcommand{\E}{{\mathcal E}}
\newcommand{\linspan}{\textrm{span}}
\newcommand{\rank}{\textup{rank}}

\newtheorem{theorem}{Theorem}
\numberwithin{theorem}{section} 
\numberwithin{equation}{section} 
\newtheorem{definition}[theorem]{Definition}
\newtheorem{lemma}[theorem]{Lemma}

\newtheorem{example}[theorem]{Example}

\newtheorem{proposition}[theorem]{Proposition}

\newcommand{\Diag}{\mbox{\rm Diag}\,}

\newcommand{\VV}{{\mathcal{V}}}
\newcommand{\R}{{\mathbb{R}}}
\newcommand{\C}{{\mathbb{C}}}

\newcommand{\f}{\frac}

\newcommand{\ra}{\rightarrow}

\newcommand{\lt}{\left}
\newcommand{\rt}{\right}

\title[Real Critical Points of the Distance to Orthogonally Invariant Sets]{Counting Real Critical Points of the Distance\\ to Orthogonally Invariant Matrix Sets}
\author{Dmitriy Drusvyatskiy \and Hon-Leung Lee \and Rekha R. Thomas}
\address{Department of Mathematics, University of Washington, Box 354350, Seattle, WA 98195-4350}
\email{[ddrusv, hllee, rrthomas]@uw.edu}
\thanks{Research of Drusvyatskiy was partially supported by the AFOSR YIP award FA9550-15-1-0237. Lee and Thomas were partially supported by the NSF grant DMS-1418728.}

\begin{document}

\begin{abstract}
Minimizing the Euclidean distance to a set arises frequently in applications. When the set is algebraic, a measure of complexity of this optimization problem is its number of critical points. In this paper we provide a general 
framework to compute and count the real smooth critical points of a data matrix on an orthogonally invariant set of matrices. The technique relies on ``transfer principles'' that allow calculations to be done in the space of singular values of the matrices in the orthogonally invariant set.  The calculations often simplify greatly and yield transparent formulas. 
We illustrate the method on several examples, and compare our results to the recently introduced notion of {\em Euclidean distance degree} of an algebraic variety.
\end{abstract}

\maketitle

\section{Introduction} \label{sect:intro}
Finding an element of a subset $\mathcal{V}$ in $\R^n$ closest to a specified point $y$ is a common task in computational mathematics, often called the {\em Euclidean distance (ED) minimization problem}:
\begin{align} \label{prob:ED} \textup{minimize } \sum_{i=1}^n (x_i-y_i)^2 \,\,\textup{ subject to } \,\,x \in \V. \end{align}
 Our current work is motivated by the systematic study of the ``critical points'' of the problem \eqref{prob:ED} in an algebraic setting, initiated in \cite{DHOST} and continued in \cite{draisma,hwangrae,genericEDdeg}. There, the basic assumption is that 
$\mathcal{V}$ is a {\em real variety} ---  zero set of finitely many polynomials with real coefficients. 
Consider now the set $\V_\CC$ of complex points satisfying the defining equations of $\mathcal{V}$.
Then a {\em critical point} of $y \in \CC^n$ with respect to \eqref{prob:ED} is any smooth (possibly complex) point
$x$ of $\V_\CC$ such that $y-x$ lies in the normal space of $\V_\CC$ at $x$, meaning that $y-x$ lies in the span of the gradients of the defining equations at $x$. It was shown in \cite{DHOST} that for a general data point $y\in \C^n$, the number of (complex) critical points of 
\eqref{prob:ED} is a constant. This constant is called the {\em Euclidean distance degree} of $\V$, denoted by EDdegree$(\V)$, and is a measure of the algebraic complexity of expressing a minimizer of \eqref{prob:ED} as a function of $y$. 

The work in this paper is geared towards understanding the (real) critical points
of {\em orthogonally invariant} matrix sets.
Our theme is best illustrated with an example. 
	Fix positive integers $r \leq n \leq t$ and consider the  matrix set 
	$$
	\RR^{n\times t}_r := \{X\in \RR^{n\times t} \ : \  \rank(X) \leq r\}.
	$$
Finding the closest matrix of rank at most $r$ to a given matrix $Y$ arises in many applications.  The set $\RR^{n \times t}_r$ is a real variety, and the authors of 
\cite{DHOST} established that $\textup{EDdegree}(\RR^{n\times t}_r) = {n \choose r}$. They also provide a recipe for  
all the critical points of $Y$ on $\RR^{n\times t}_r$, which may be viewed as a generalization of the {\em Eckart-Young theorem}.

In the context of this paper, what is important about $\RR^{n\times t}_r$ is that it is {\em orthogonally invariant}, meaning that if $X \in \RR^{n \times t}_r$ then for all orthogonal matrices $U$ and $V$ of appropriate sizes, 
$UXV^\top$ is also in $\RR^{n \times t}_r$. Alternately, membership of $X$ in the set $\RR^{n\times t}_r$ 
is determined solely by the vector of singular values $\sigma(X) = (\sigma_1(X),\ldots,\sigma_n(X))$. 
Indeed, a matrix $X$ lies in $\RR^{n\times t}_r$ if and only if its vector of singular values $\sigma(X)$ lies in the set
$$\R^n_r:=\{x\in\R^n: \rank(x)\leq r\},$$
where $\rank(x)$ denotes the number of nonzero coordinates of $x$. Observe that the geometry of the piecewise linear set $\R^n_r$ is much simpler than that of the highly nonlinear set $\R^{n\times t}_r$. In particular, $\RR^n_r$ is also a variety and $\textup{EDdegree}(\RR^{n}_r) = {n \choose r}$. 
The equality $\textup{EDdegree}(\RR^{n}_r)=\textup{EDdegree}(\RR^{n\times t}_r)$ is not accidental; 
in this paper, we elucidate this phenomenon. 

As alluded to above, our focus in this paper is on a (naturally defined) real analog of critical points for \eqref{prob:ED}. Namely, we say that $x\in\R^n$ is a (real) {\em ED critical point} of $y\in\R^n$ relative to a set $\mathcal{V}\subset\R^n$ (not necessarily a variety) if $x$ is a smooth point of $\mathcal{V}$ and $y-x$ lies in the normal space to $\mathcal{V}$ at $x$; see Definition~\ref{def:real critical points} for details. 
Our main result shows that one can always obtain the ED critical points of an orthogonally invariant matrix set by restricting to diagonal matrices, or equivalently, to an {\em absolutely symmetric} set obtained from the singular values of the matrices.
In the last section, we  explore the connection between ED critical points and the critical points in the sense of \cite{DHOST} that happen to be real, a surprisingly subtle topic. In particular, ED critical point calculations often help to understand $\textup{EDdegree}(\mathcal{V})$, when $\mathcal{V}$ is a  variety.

Sets constrained via their singular values are numerous in applications. 
Define $\E \subseteq \RR^{3 \times 3}$ to be the set of rank deficient matrices with two equal singular values:
\begin{equation}\label{eqn:essential}
\mathcal{\E} :=\{X\in\R^{3\times 3}: \sigma_1(X)=\sigma_2(X), \,\sigma_3(X)=0\}.
\end{equation}
 A matrix $X \in \E$ is called an {\em essential matrix} in 3D computer vision and represents a pair of calibrated pinhole cameras \cite[Chapter 9]{hartley-zisserman}. The set $\E$ also happens to be a real variety cut out by the following ten cubic equations in the entries of $X$\cite[Proposition 4]{maybank}:
		\begin{align} \label{demazure cubics}
			\det(X) = 0, \,\,\, 2 XX^\top X - {\rm tr}(XX^\top) X = 0.
		\end{align}
Observe that $\E$ is orthogonally invariant as membership of a matrix $X$ in $\E$ depends only on its singular values. We will show that 
one can obtain the ED critical points of $\E$ from the simpler set $E_{3,2}$ of vectors in $\R^3$ with one coordinate zero and the other two equal in absolute value. 

The idea of studying orthogonally invariant matrix sets $ \M \subseteq \R^{n\times t}$ via their diagonal restrictions $S=\{x:\Diag x\in \M\}$ is not new, and goes back at least to von Neumann's theorem on {\em unitarily invariant matrix norms} \cite{von_Neumann}. In recent years, the general theme has become clear:
 various analytic properties of $\M$ and $S$ are in one-to-one correspondence. This philosophy is sometimes called the ``transfer principle''; see for instance, \cite{man}. For example, $\M$ is $C^p$-smooth around a matrix $X$ if and only if $S$ is $C^p$-smooth around $\sigma(X)$ \cite{spec_id,der,diff_2,diff_1}. Other properties, such as {\em convexity} \cite{cov_orig}, {\em positive reach} \cite{spec_prox}, {\em partial smoothness} \cite{spec_id}, and {\em Whitney conditions} \cite{mather} follow the same paradigm. We note in passing that the setting of eigenvalue constrained sets of nonsymmetric matrices is more complicated; see e.g. \cite{math_e-val,nonlip_eig}.  In the current work, we 
 derive a transfer theorem for the ED critical points of an orthogonally invariant matrix set $\M$ and illustrate it on several examples. The manuscripts \cite{send,eval,spec_prox} play a central role in our work. The transfer paradigm has many useful features. The first is that in many instances, the calculation over the set $S \subseteq \RR^n$ is simpler than the original one over $\M$. This can make the formulas for the number of ED critical points of $\M$ much more transparent as compared to the direct calculation in matrix space. The case of $\RR^{n\times t}_r$ is an example of this. Secondly, our method focuses on (real) ED critical points as opposed to all complex critical points in the setting of algebraic varieties. This is useful in applications and allows us to gauge the difference between the real and complex situations (under appropriate conditions). 

This paper is structured as follows. In Section~\ref{sect:ED critical points} we establish basic notation and define the key notion of ED critical points of a set in $\RR^n$ with respect to a data point $y$. Then in Section~\ref{sect:real critical points of spectral sets} we derive our main theorem 
that transfers the study of ED critical points of an orthogonally invariant set of matrices to that of its diagonal restriction.
Section~\ref{sect:applications} illustrates the technique on some concrete examples deriving formulas for the number of ED critical points in each case. Finally in Section~\ref{sect:eddegree} we establish the relationships between our work and the results in \cite{DHOST}, by restricting to algebraic varieties. In particular, we compare our formulas to those for EDdegree in several instances. We note in passing, that many of the results in the paper also hold for
complex matrices with respect to the standard Hermitian inner product.

\section{ED critical points of subsets of $\RR^n$} \label{sect:ED critical points}
Throughout this paper we consider the $n$-dimensional Euclidean space $\R^n$, with a fixed orthonormal basis.
Let $\langle\cdot,\cdot \rangle$ denote the inner product in this setting, and $\|\cdot\|$ denote the induced norm. The {\em distance} and the {\em projection} of a (data) point $y \in \RR^n$ onto a subset $S \subseteq \R^n$, are defined by 
 \begin{align*}
 {\rm dist}_S(y)&:= \inf_{x\in S} \|y-x\|, \qquad\textup{ and } \\
 {\rm proj}_S(y) &:= \{x \in S\, : \, {\rm dist}_S(y) = \|y-x\|\}.
 \end{align*}
Computing the distance of $y$ to $S$  amounts to solving the optimization problem:
\begin{align} \label{eq:distance prob}
\inf_{x\in S} \frac{1}{2}\|y-x\|^2 = \inf_{x\in S} \frac{1}{2} \sum_{i=1}^n (y_i-x_i)^2.
\end{align}
A classical first-order necessary condition for a putative 
point $x\in S$ to be optimal for \eqref{eq:distance prob} is that the gradient of the object function, namely $x-y$, makes an acute angle with every vector $v$ in the {\em tangent cone}\footnote{If $x$ is an isolated point of $S$, then $\mathcal{T}_{S}(x)$ is declared to consist only of the origin.} $$\mathcal{T}_{S}(x):=\R_+\Big\{\lim_{z_i\to x} \frac{z_i-x}{\|z_i-x\|}: z_i\in S \Big\}.$$
One can regard such points as generalized critical points of the distance minimization problem \eqref{eq:distance prob}. On the other hand, in order to compare and unify our work with that in \cite{DHOST},  we will impose an extra smoothness condition on a point $x \in S$ in order for it to be considered critical for \eqref{eq:distance prob}. To this end, throughout the manuscript, we fix $p\in \{2, 3, \ldots, \infty, \omega\}$, and say that a point $x \in S$ is {\em $C^p$-smooth} if 
there is a neighborhood $\Omega$ of $x$ such that 
$S\cap \Omega$ is an embedded $C^p$-smooth manifold 
(recall that $C^\omega$ means {\em real analytic}). In this case, the tangent cone $\mathcal{T}_{S}(x)$ is the usual tangent space in the sense of differentiable manifolds, and the criticality condition above amounts to the inclusion $y-x\in \N_{S}(x)$, where $\N_{S}(x)$ denotes the {\em normal space} to $S$ at $x$ --- the orthogonal complement of $\mathcal{T}_{S}(x)$.

From now on we abbreviate ``$C^p$-smooth"  to ``smooth". We will use $S^*$ to denote the set of 
all smooth points of $S$. Here is then our main definition, attuned to the one considered in \cite{DHOST}.

\begin{definition}[ED critical points] \label{def:real critical points}
	{\rm 
Consider a set $S \subseteq \RR^n$ and a point $y$ in $\RR^n$. A point $x \in S$ is an {\em ED critical point} of $y$ on $S$ if the following conditions hold:
\begin{enumerate}[(i)]
\item $x\in S^*$, and 
\item $y-x \in \N_S(x)$. 
\end{enumerate}
The symbol $\textup{C}_S(y)$ will denote the set of all ED critical points of $y$ on $S$, while the cardinality of $\textup{C}_S(y)$ will be denoted by $C^\#_S(y)$. 
} 
\end{definition}

By definition, all ED critical points of $y$ on $S$ are real and smooth.
The number of ED critical points $C^\#_S(y)$ varies with $y$. For example, consider the parabola $S$ shown in 
Figure~\ref{fig:parabola}, with an additional curve called its {\em ED discriminant} or {\em evolute}. 
All points $y$ above the evolute have three  ED critical points while below the evolute they have one  ED critical point.
Since $S$ is an algebraic variety, we can compute its EDdegree which is three, i.e., $S_\CC$ has 
three distinct regular complex critical points almost everywhere. We will comment more on the ED discriminant in 
the Appendix.

\begin{figure} 
\includegraphics[scale=0.3]{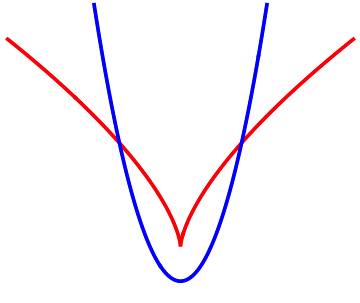}
\caption{A parabola in the plane with its evolute or ED discriminant.} \label{fig:parabola}
\end{figure}

Letting $\Theta$ be the collection of all Lebesgue null subsets of $\R^n$, the following worst-case measure of criticality arises naturally:
$$
C^\#(S) := \inf_{\Gamma\in \Theta } \sup_{y\in \Gamma^c} C^\#_S(y),
$$
where $\Gamma^c$ is the set complement of $\Gamma$ in the ambient space. 
Indeed, intuitively if one believes that any single zero measure set $\Gamma$ can be discarded, then $C^\#(S)$ measures the maximal value of $C^\#_S(y)$ that has a non-negligible chance of being encountered. For our purposes, one could think of $C^\#(S)$ as a real analog of 
the EDdegree considered in \cite{DHOST}. We discuss this further in Section~\ref{sect:eddegree}.

We note in passing that this criticality measure can be infinite in pathological situations (e.g. union of countably many co-centric circles in $\R^2$). On the other hand, for structured sets, such as those that are semi-algebraic, this number is finite. 

\begin{proposition}[ED critical points of semi-algebraic sets]  \label{prop:semialgebraic}
For any semi-algebraic set $S$ in $\R^n$, the number $C^\#(S)$ is finite.	
\end{proposition}
\begin{proof}
Standard quantifier elimination shows that the set $S^*$ is semi-algebraic. Define the manifolds $M_i:=\{x\in S^*: \dim \mathcal{T}_{S}(x)=i\}$ for $i=1,\ldots,n$. Again, quantifier elimination shows that the normal bundle 
$$\Omega :=\bigsqcup_{i} \{(x,v): x\in M_i,\, v\in \mathcal{N}_{M_i}(x)\}$$
is a semi-algebraic subset of $\R^n\times\R^n$ having dimension $n$. Consider now the mapping $\phi\colon\Omega\to\R^n$ defined by $\phi(x,v)= x+v$. Notice that $C_S(y)$ coincides with the projection of the preimage $\phi^{-1}(y)$ onto $x$. For dimensional reasons, there is a full-measure set $D\subseteq\R^n$ so that for every $y\in D$ the preimage $\phi^{-1}(y)$ has finite cardinality. Moreover, since preimages of a semi-algebraic map have a uniformly bounded number of connected components (see e.g. 
\cite[Theorem 3.12]{minimal}), the quantity $C^{\#}_S(y)$ is uniformly bounded on $D$. The result follows.  
\end{proof}


Next we consider ED critical points on a union of finitely many sets.
Consider a finite collection of sets $\{S_i\}_{i\in \mathcal{I}}$ in $\R^n$ and define the union $U:=\bigcup_{i\in \mathcal{I}} S_i$. In general, the two sets  $C_{U}(y)$ and $\bigcup_{i\in \mathcal{I}} C_{S_i}(y)$ can be vastly different because of the way the sets $S_i$ intersect.  For instance, think of two half spaces whose union is $\RR^n$. A simple situation in which more can be said is when locally around each critical point $x\in C_U(y)$, the set $U$ coincides with $S_i$ for some $i\in \mathcal{I}$. In that situation, if $x \in C_U(y)$ then $x$ is also in $C_{S_i}(y)$.
An important situation in this paper is the case of all $S_i$ being affine subspaces. We say that a finite collection of sets $\{S_i\}_{i\in \mathcal{I}}$ in $\R^n$ is {\em minimally defined} if no $S_i$ is contained in any $S_j$ for distinct indices $i$ and $j$.

\begin{proposition}[ED critical points of affine complexes] \label{prop:linear}
Consider a finite collection of affine subspaces $\{S_i\}_{i\in \mathcal{I}}$ in $\RR^n$, that is minimally defined, and let  $U:=\bigcup_{i\in \mathcal{I}} S_i$. Then we have
\begin{equation}\label{eqn:crit eq}
U^*=\{x\in\R^n: \textrm{ there exists unique }i\in \mathcal{I} \textrm{ with } x\in S_i\}.
\end{equation}
Consequently for any $y\in \R^n$, we have 
\begin{align} \label{eq:crit subspace}
{\rm C}_{U}(y) = 
\bigsqcup_{i\in \mathcal{I}} \lt({\rm proj}_{S_i} (y) \cap U^* \rt),
\end{align}
and the equality
$$C^\#(U)= |\mathcal{I}|.$$
\end{proposition}

\begin{proof}
The inclusion $\supseteq$ in  \eqref{eqn:crit eq} follows since the sets $S_i$ are affine. To see the reverse inclusion, observe that the tangent cone to $U$ at any point $x$ coincides with the union 
$\bigcup_{i\in \mathcal{I}:\, x\in S_i} (S_i-x)$. For $x \in U^*$, the cone $\mathcal{T}_U(x)$ is itself a linear subspace, and hence, 
by the minimality of the collection $\{S_i\}_{i\in \mathcal{I}}$,  there exists a unique index $i\in \mathcal{I}$ satisfying $x\in S_i$. This establishes \eqref{eqn:crit eq}. Equation \eqref{eq:crit subspace} then follows immediately.

To see the last claim, 
consider the set 
$$Z :=\{y\in\R^n: \exists i \in\mathcal{I} \textrm{ with }{\rm proj}_{S_i} (y)\cap U^*=\emptyset\}.$$
We will show that $Z$ is a finite union of proper affine subspaces of $\R^n$. To see this, consider 
a point $y\in Z$ along with an index $i\in\mathcal{I}$ satisfying 
${\rm proj}_{S_i} (y)\cap U^*=\emptyset$. From \eqref{eqn:crit eq}, we conclude that there exists $j\in \mathcal{I}$, distinct from $i$, satisfying ${\rm proj}_{S_i} (y)\cap S_{j}\neq \emptyset$. Thus $y$ lies in the set $ S_i^{\perp}+ (S_i\cap S_{j})$. Since the collection $\{S_i\}$ is minimally defined, the intersection $S_i\cap S_{j}$ has dimension strictly smaller than that of $S_i$. Consequently the affine space  $ S_i^{\perp}+(S_i\cap S_{j})$ has dimension strictly smaller than $n$. Taking the union over all pairs of distinct indices $i,j\in\mathcal{I}$, we deduce that $Z$ is a finite union of proper affine subspaces of $\R^n$. Therefore, by \eqref{eq:crit subspace}, $C^\#_U(y)=|\mathcal{I}|$ for all $y \not \in Z$ which proves that $C^\#(U)= |\mathcal{I}|$.
\end{proof}

The following two elementary examples illustrate Proposition~\ref{prop:linear}; these are essentially the piecewise linear examples 
alluded to in the introduction. In the next section, we will use them
to obtain the ED critical points of the (nonlinear) matrix sets $\R^{n\times t}_r$ and $\E$. In what follows we set $[n]  := \{1,\ldots,n\}$.

\begin{example}[Union of $r$-dimensional coordinate subspaces] \label{ex:rank}
	{\rm 
		For $x\in\RR^n$, define the  rank of $x$, denoted by ${\rm rank}(x)$, to be the number of nonzero coordinates of $x$.
		Fix an integer $r \in [n]$ and recall the set 
		\begin{equation} \label{eq:def Rnr}
			\RR^n_r = \{ x \in \RR^n \,:\,    \rank(x)\leq r \}.
		\end{equation}
		Define $\mathcal{I}$ to be the collection of distinct cardinality $r$ subsets of $[n]$. Then
		 $|\mathcal{I}|={n \choose r}$, and 
		$$\RR^n_r=\bigcup_{S\in\mathcal{I}} \linspan\, \{e_i\}_{i\in S},$$
		where $e_i$ denotes the $i$'th coordinate vector in $\RR^n$.
	 This representation of $\RR^n_r$ is minimal and therefore,  Proposition \ref{prop:linear} implies the equality
			 $C^\#(\RR^n_r) = |\mathcal{I}| = {n \choose r}.$}
\end{example}

\begin{example}[$k$ nonzero entries equal up to sign] \label{ex:Enk}
	{\rm 
		Fix an integer $k\in [n]$, and
		let $E_{n,k}$ be the set of points in $\RR^n$ with the property that $k$ of their coordinates are equal in absolute value and the other $n-k$ coordinates are zero. Note that $E_{n,k}$ can be written as the union of $2^{k-1} {n \choose k}$ linear subspaces (minimally defined). Proposition \ref{prop:linear} then implies that 
		$C^\#(E_{n,k}) = 2^{k-1} {n \choose k}.$
                     As a special case, the set $E_{3,2}$, which is a union of six lines in $\RR^3$ (see Figure  \ref{fig:e32}), 
                     satisfies $C^\#(E_{3,2}) = 6$.	}
\end{example}

\begin{figure}
\includegraphics[scale = 0.5]{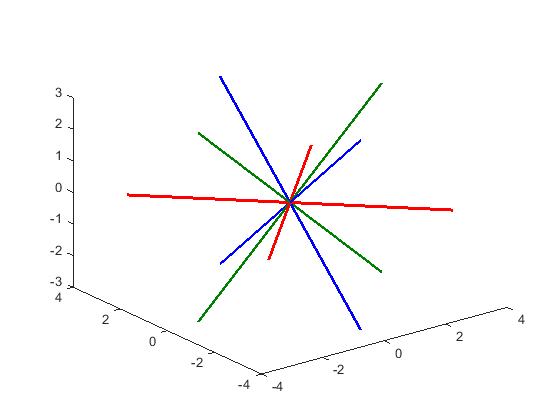}
\caption{the set $E_{3,2}$} \label{fig:e32}
\end{figure}

\section{ED critical points of orthogonally invariant matrix sets}  \label{sect:real critical points of spectral sets}

In this section we describe our main result which yields an elegant technique for counting the ED critical points of orthogonally invariant matrix sets, based on the tools established in \cite{send,eval,spec_prox}. Setting the notation, let $\R^{n\times t}$ denote the set of real $n\times t$ matrices, where we assume
without loss of generality that $n \leq t$.
The singular value map $\sigma\colon\R^{n\times t}\to\R^n$ assigns to each matrix $X \in \RR^{n \times t}$ the vector of its singular values $\sigma(X) := (\sigma_1(X), \ldots, \sigma_n(X))$ arranged in non-increasing order. The corresponding inverse map is defined as 
$$\sigma^{-1}(S):=\{X\in \R^{n\times t}: \sigma(X)\in S\} \qquad\textrm{ for any subset } S \textrm{ of } \R^n.$$

We will be interested in subsets of $\RR^{n \times t}$ that are invariant under multiplication on the left and right by orthogonal matrices. In what follows, we let $\mathcal{O}^s$ denote the group of real $s\times s$ orthogonal matrices.

\begin{definition}[Orthogonal invariance]{\hfill \\ }
	{\rm
		A set $\M \subseteq \RR^{n \times t}$ is {\em orthogonally invariant} if 
		$$\M= U \mathcal{M} V^\top  \qquad\textrm{ for all }  U\in \mathcal{O}^n \textrm{ and } V\in \mathcal{O}^t.$$}
\end{definition}

\vspace{0.15cm}

An example of an orthogonally invariant matrix set is the essential variety 
\begin{equation*} 
\E=\{X\in \R^{3\times 3} \ : \  \sigma_1(X)=\sigma_2(X),\, \sigma_3(X)=0\}
\end{equation*}
considered in the introduction.
At first sight, $\E$ is a complicated set; it is a highly nonlinear real variety cut out by the cubic polynomials in  
\eqref{demazure cubics}. 
Since orthogonally invariant matrix sets are precisely those matrix sets for which membership is determined solely by the singular values of its elements, the 
restriction of such sets to the subspace of diagonal matrices plays an important role. Our strategy will be to exploit this observation to calculate the ED critical points of orthogonally invariant matrix sets.

A mapping $\pi: [n] \ra \{\pm 1, \ldots, \pm n\}$ is a {\em signed permutation} 
if  the assignment $i\mapsto|\pi(i)|$  is a permutation on  $[n]$ in the usual sense.
We let $\Pi_n^{\pm}$ denote the set of signed permutations on $[n]$. 
Note that any signed permutation $\pi\in \Pi_n^{\pm}$ induces a linear map $\RR^n \ra \RR^n$ which we also denote by $\pi$.

\begin{definition}[Absolute symmetry]
{\rm
A set $S \subseteq \RR^n$ is  said to be {\em absolutely symmetric} if $$S = \pi S\qquad \textrm{ for all  } \pi\in\Pi_n^{\pm}.$$ 
For any set $S \subseteq \RR^n$, we call $\Pi_n^{\pm}  S:=\{ \pi x \,:\, \pi \in \Pi_n^\pm \textup{ and } x \in S \} $ the {\em absolute symmetrization} of $S$.}
\end{definition}

For a vector $x \in \RR^n$, let  $\Diag x \in \RR^{n \times t}$ denote the matrix with $x$ in its principal diagonal 
and zeros elsewhere. 
If $\M \subseteq \RR^{n \times t}$ is orthogonally invariant, then the set $\{x\in\R^n: \Diag x\in \M\}$ is absolutely symmetric, 
and we have the following basic observation (see e.g.  \cite[Proposition 5.1]{send}).

\begin{theorem} [Diagonal Correspondence] \label{thm:sets}
	A set $\M \subseteq \RR^{n \times t}$ is orthogonally invariant if and only if there exists an 
	absolutely symmetric set $S \subseteq \RR^n$ such that 
	$\M = \sigma^{-1}(S)$.
\end{theorem}

Thus the assignment $\sigma^{-1}$ from the family of absolutely symmetric sets in $\RR^n$ to the
family of orthogonally invariant sets in $\RR^{n\times t}$ is a bijection. 
As discussed in the introduction, it so happens that various analytic properties of absolutely symmetric sets $S$ and orthogonally invariant sets $\sigma^{-1}(S)$ are in one-to-one correspondence; see e.g. \cite{man,von_Neumann,der,diff_2,diff_1,cov_orig,spec_prox,spec_id,mather}. This may seem somewhat surprising since $\sigma$ is a highly nonsmooth mapping; it is the absolute symmetry of the underlying set $S$ that makes up for the fact. As an illustration, we show in Theorem~\ref{thm:variety} how the transfer principle for algebraicity can be used to see 
that the set of essential matrices $\E$  is a variety, even without explicitly knowing its defining equations  \eqref{demazure cubics}. The proof is entirely analogous to the symmetric case in \cite[Proposition 1.1]{man}; we provide an argument for completeness.

Recall that a set $\V \subseteq \RR^n$ is called a {\em real variety} if there exist polynomials $f_1, \ldots, f_s \in \RR[x_1,\ldots,x_n]$ such that $\V = \{ x \in \RR^n \,:\, f_1(x) = \ldots = f_s(x) = 0 \}$. Note that if $x \in \V$, then $\sum_{i=1}^{s} f_i^2(x) = 0$, and conversely if a point $x \in \RR^n$ satisfies $\sum_{i=1}^{s} f_i^2(x) = 0$, then $f_i(x) = 0$ for all $i=1,\ldots,s$, and hence, $x \in \V$. Therefore, $\V$ can be described as the set of real zeros of a single sum of squares polynomial with real coefficients, which we call a defining polynomial of $\V$.

\begin{theorem}[Transfer of algebraicity] \label{thm:variety}
	Suppose $S \subseteq \RR^n$ is an absolutely symmetric set. Then $\sigma^{-1}(S) \subseteq \RR^{n \times t}$ is a real variety if and only if $S$ is a real variety.
\end{theorem}
\begin{proof}
	If $\sigma^{-1}(S)$ is a real variety, then 
	 $S=\{x\in\R^n: \Diag x\in \sigma^{-1}(S)\}$ is the real variety cut out by the equations defining 
	 $\sigma^{-1}(S)$ when $x_{ij}$ are set to zero for all $i \neq j$.
	Suppose conversely that $S$ is a real variety with defining polynomial $f \in \RR[x_1,\ldots,x_n]$.
	Consider the polynomial $\hat{f}(x) :=\sum_{\pi \in \Pi_n^{\pm}} f^2(\pi x)$. 
	Note that $\hat{f}(x)$ is also a defining polynomial of $S$, and hence, 
	$\sigma^{-1}(S)$ is the zero level set of $\hat{f} \circ \sigma$. To finish the proof, we just need to show that 
	$\hat{f}\circ\sigma(X)$ is a polynomial in the entries of $X$.
	To this end, since $\hat{f}$ is invariant under sign changes, it is easy to see that $\hat{f}$ is a symmetric polynomial in the squares $x^2_1,\ldots, x^2_n$, that is we may write $\hat{f}(x)=g(x^2_1,\ldots, x^2_n)$ for some symmetric polynomial $g$. By the fundamental theorem of symmetric polynomials (see e.g. \cite{symm}) we may write $g$ as a polynomial of elementary symmetric polynomials $\epsilon_1,\ldots, \epsilon_n$. On the other hand, the expressions  $\epsilon_i(\sigma^2_1(X),\ldots, \sigma^2_n(X))$ coincide with the coefficients of the characteristic polynomial of $X^\top X$ and are hence, polynomial expressions in the entries of $X$.  
\end{proof}

Another useful illustration of transfer principles concerns the distance to orthogonally invariant matrix sets. 
Recall that for a set $S \subseteq \RR^n$, the {\em distance} and {\em projection} of a point $y \in \RR^n$ 
to (respectively, onto) $S$ are defined by
${\rm dist}_S(y):= \inf_{x\in S} \|y-x\|,$ and  
${\rm proj}_S(y) := \{x \in S \ : \  {\rm dist}_S(y) = \|y-x\|\}$. (The distance and the projection in the matrix space $\RR^{n\times t}$ are defined analogously with respect to the {\em Frobenius norm} $\|Y\|:=\sqrt{\sum_{i,j} Y^2_{ij}}$.)
Then for an absolutely symmetric set 
$S \subseteq  \R^n$, the following holds \cite[Proposition 8]{spec_prox}:
\begin{equation}\label{eqn:dist_est}
{\rm dist}_{\sigma^{-1}(S)}(Y) = {\rm dist}_S(\sigma(Y)).
\end{equation}
This in turn implies the following result, which was essentially proved in \cite[Proposition 8]{spec_prox}, though not formally recorded. We provide a proof sketch for completeness.

\begin{proposition}[Projections onto orthogonally invariant matrix sets] \label{prop:proj transfer}
If $S \subseteq \RR^n$ is an absolutely symmetric set, then for any matrix $Y\in\RR^{n\times t}$ , the projection ${\rm proj}_{\sigma^{-1}(S)}(Y)$ is precisely the 
set 
\begin{align} \label{eq:spectral proj}
\left\{ U (\Diag  x) V^\top \,:\, U\in \mathcal{O}^n, V\in\mathcal{O}^t \textup{ where }
\begin{array}{l} 
Y = U (\textup{Diag } \sigma(Y))V^\top,\\
x \in {\rm proj}_S(\sigma(Y)) 
\end{array}
\right\}.
\end{align}
\end{proposition}

\begin{proof}
Consider first matrices $U\in \mathcal{O}^n, V\in\mathcal{O}^t$ with $Y = U (\textup{Diag } \sigma(Y))V^\top$ and a vector $x\in {\rm proj}_S(\sigma(Y))$. Define $X:=U(\Diag x)V^\top$ and observe the equalities:
\begin{align*}
\|X-Y\|=\|x-\sigma(Y)\|
={\rm dist}_S(\sigma(Y))={\rm dist}_{\sigma^{-1}(S)}(Y),
\end{align*}
where the last equality follows from \eqref{eqn:dist_est}.
Hence the inclusion $X\in {\rm proj}_{\sigma^{-1}(S)}(Y)$ is valid, as claimed.	
Conversely, for any matrix $X\in {\rm proj}_{\sigma^{-1}(S)}(Y)$ observe
$${\rm dist}_{\sigma^{-1}(S)}(Y)= \|X-Y\|\geq  \|\sigma(X)-\sigma(Y)\|\geq {\rm dist}_S(\sigma(Y)),$$
where the first inequality follows from the Von Neumann's trace inequality ~\cite[p. 182]{HJ}; see also \cite[Theorem 4.6]{send}. 
Equation  \eqref{eqn:dist_est} implies equality throughout. 
In particular we get $\sigma(X)\in {\rm proj}_S(\sigma(Y))$. 
Moreover, 
applying the equality characterization in the trace inequality \cite[Theorem 4.6]{send}, we conclude that $X$ and $Y$ admit a simultaneous ordered singular value decomposition, that is, 
there exist matrices $U\in \mathcal{O}^n, V\in\mathcal{O}^t$ with $Y = U (\textup{Diag } \sigma(Y))V^\top$ and $X = U (\textup{Diag } \sigma(X))V^\top$. 
This completes the proof. 
\end{proof}

A consequence of Proposition~\ref{prop:proj transfer}  is a convenient representation of the normal space at a point in an orthogonally invariant manifold. First of, the set $\sigma^{-1}(S)$ is smooth around $X$ if and only if $S$ is smooth around $\sigma(X)$ (see e.g. \cite[Theorem 2.4]{spec_id}). Next recall that at any $C^2$-smooth point $x$ of a set $M$, the following equivalence holds:
$$z\in \mathcal{N}_{M}(x) \quad\Longleftrightarrow\quad \{x\}= {\rm proj}_{M}(x+\lambda z) \quad \textrm{ for some } \quad \lambda>0.$$
Proposition~\ref{prop:proj transfer} and a short computation implies 
that at any smooth point $X$ of $\sigma^{-1}(S)$ the following formula holds:
$$\N_{\sigma^{-1}(S)}(X) = \left\{ U (\Diag  z) V^\top \,:\, U\in \mathcal{O}^n, V\in\mathcal{O}^t \textup{ with }
\begin{array}{l} 
X = U (\Diag \sigma(X))V^\top,\\
z \in \N_S(\sigma(X)) 
\end{array}
\right\}.$$
A more general expression without smoothness assumptions can be found in \cite[Theorem 7.1]{send}. 
Using the above facts, we now derive a transfer principle for ED critical points which is the main result of this section. The proof uses the following lemma.

\begin{lemma}[Transfer of Lebesgue null sets]\label{lem:meas}
	Consider an absolutely symmetric set $S\subseteq\R^n$. Then $S$ is Lebesgue null if and only if $\sigma^{-1}(S)$ is Lebesgue null.	
\end{lemma}

We have placed the proof of the lemma above at the end of the section so as to not stray from the narrative.

\begin{theorem}{\rm ({ED critical points of orthogonally invariant matrix sets})}  \label{thm:spectral_calculus} {\hfill \\ }
	Consider an absolutely symmetric set $S\subseteq \R^n$ and a matrix 
	$Y\in\R^{n\times t}$ along with a singular value decomposition $Y=\overline{U}(\Diag \sigma(Y))\overline{V}^\top$. 
	Suppose moreover, that $Y$ has all distinct singular values. 
	Then for the orthogonally invariant set $\M:=\sigma^{-1}(S)$, we have
\begin{equation}\label{eqn:crit_corr}
			{\rm C}_{\M}(Y)=\{\overline{U} \big(\Diag\omega\big)\overline{V}^\top\ \colon  \ 
	\omega\in {\rm C}_{S}(\sigma(Y))
	\}.
	\end{equation}
	Consequently, we also obtain the equality $C^\#(\mathcal{M})=C^\#(\mathcal{S})$.
\end{theorem}

\begin{proof} 
	Consider a critical point $X\in {\rm C}_{\M}(Y)$. Then there exist orthogonal matrices 
	$U$ and $V$, and a vector $z\in \N_{S}(\sigma(X))$ satisfying 
	$$X=U (\Diag \sigma(X)) V^\top \quad \textrm{ and }\quad Y=X+U (\Diag z) V^\top.$$
	Hence we deduce
	$$Y =U (\Diag (\sigma(X) +z))V^\top.$$
	By uniqueness of singular values and singular vectors, there exists 
	a signed permutation matrix $\pi\colon\R^n\to\R^n$ satisfying $$\sigma(Y)=\pi(\sigma(X)) +\pi(z).$$
	On the other hand, it is easy to verify the inclusion 
	$\pi(z)\in\N_{S}(\pi \sigma(X))$.
	Therefore, we obtain $\pi(\sigma(X))\in C_S(\sigma(Y))$.
	
	Define now the extended signed permutation matrix $\hat{\pi}:=\left[\begin{matrix}
	\pi & 0\\
	0 & I 
	\end{matrix}\right] \in \R^{t\times t}$.  
	Observe then the equality 
	$$Y =U (\Diag \pi^{-1}(\sigma(Y)))V^\top=(U\pi) (\Diag \sigma(Y))(V\hat{\pi})^\top.$$
	Changing the left and the right singular vector pairs in $U$ and in $V$ by the same sign change in the first place, if needed, we can ensure $U\pi=\overline{U}$ and $V\hat{\pi}=\overline{V}$. Hence,
	$$X=U (\Diag \sigma(X)) V^\top=\overline{U} (\Diag \pi(\sigma(X))) \overline{V}^\top,$$
	as claimed.

	Now to prove the reverse inclusion, consider a point $\omega\in {\rm C}_S(\sigma(Y))$. 
	Thus there exists a vector $z\in \N_S(\omega)$ satisfying
	$$\sigma(Y)-\omega=z.$$
	Choose a signed permutation $\pi$ so that $\pi(\omega)$ is nonnegative and nonincreasing. Observe then that
	$$\pi(\sigma(Y))-\pi(\omega)=\pi(z)\in \N_S(\pi(\omega)).$$
	Defining $$X:=\overline{U} (\Diag \omega) \overline{V}^\top=(\overline{U}\pi) (\Diag \pi(\omega)) (\overline{V}\hat{\pi})^\top,$$ we deduce
	$$X+(\overline{U}\pi) (\Diag \pi(z)) (\overline{V}\hat{\pi})^\top=(\overline{U}\pi)  (\Diag\pi(\sigma(Y))) (\overline{V}\hat{\pi})^\top=Y.$$
	The inclusion $X\in {\rm C}_S(Y)$ follows. This establishes equation \eqref{eqn:crit_corr}.
	
	Define now the set 
	$$\Delta:=\{x\in\R^n: |x_i|=|x_j| \textrm{ for some indices } i\neq j\}.$$
	Consider a Lebesgue null set $\Gamma\subseteq\R^n$. Then clearly the absolutely symmetric set $\widehat{\Gamma}:=\Delta\cup\Pi^{\pm}_n\Gamma$ is Lebesgue null as well. Lemma~\ref{lem:meas} shows that $\sigma^{-1}(\widehat{\Gamma})$ is also Lebesgue null. 
Note  for all $\pi\in\Pi^{\pm}_{n}$ and $y\in\R^n$, one has $C_S(\pi y)=\pi C_S(y)$ and hence 
$C^\#_S(\pi y) = C^\#_S(y)$. Combining these observations, we obtain
	$$C^\#(\mathcal{M})\leq \sup_{Y\in \sigma^{-1}({\widehat{\Gamma}}^c)} C^\#_{\mathcal{M}}(Y)=\sup_{y\in {\widehat{\Gamma}}^c} C^\#_{S}(y)\leq \sup_{y\in \Gamma^c} C^\#_{S}(y),$$
	where the middle equality follows from equation \eqref{eqn:crit_corr}.
	Since $\Gamma$ was an arbitrary null set, we deduce $C^\#(\mathcal{M})\leq C^\#(S)$. 
	
	To see the reverse inequality, consider a Lebesgue null set $\Gamma\subseteq\R^{n\times t}$. 
           Consider the following subset of $\Gamma$: $$\widehat{\Gamma}:=\{Y\in \Gamma : UYV^\top\in \Gamma \textrm{ for all } U\in\mathcal{O}^n, V\in\mathcal{O}^t\}.$$
	Clearly $\widehat{\Gamma}$ is an orthogonally invariant Lebesgue null set.
	Now trivially we have
	$$\sup_{Y\in\widehat{\Gamma}^c} C^\#_{\mathcal{M}}(Y)\geq \sup_{Y\in\Gamma^c} C^\#_{\mathcal{M}}(Y).$$
	On the other hand, it is easy to verify the equation $C_{\mathcal{M}}(UYV^\top)=U(C_{\mathcal{M}}(Y))V^\top$ 
and thus, $C^\#_{\mathcal{M}}(UYV^\top) = C^\#_{\mathcal{M}}(Y)$ for all $Y\in\R^{n\times t}$ and $U\in\mathcal{O}^n,V\in\mathcal{O}^t$. Hence 
	$$\sup_{Y\in\widehat{\Gamma}^c} C^\#_{\mathcal{M}}(Y) =\sup_{Y\in\Gamma^c} C^\#_{\mathcal{M}}(Y).$$
	Notice $\widetilde{\Gamma}:=\widehat{\Gamma}\cup \sigma^{-1}(\Delta)$ is also a Lebesgue null orthogonally invariant set. Hence we deduce
	$$\sup_{Y\in\Gamma^c} C^\#_{\mathcal{M}}(Y)\geq \sup_{Y\in \widetilde{\Gamma}^c} C^\#_{\mathcal{M}}(Y)=\sup_{y\in \sigma(\widetilde{\Gamma})^c} C^\#_{S}(y)\geq C^\#(S),$$
	where the middle equality follows from equation 
	\eqref{eqn:crit_corr} and last inequality follows from Lemma~\ref{lem:meas}. Since $\Gamma$ was an arbitrary Lebesgue null set, we deduce the reverse inequality $C^\#(\mathcal{M})\geq C^\#(S)$, and hence equality, as claimed.
\end{proof}

The assumption that singular values of the data point $Y$ are distinct is essential 
 in Theorem~\ref{thm:spectral_calculus} as the following example shows.

\begin{example}[Nondistinct singular values] 
	{\rm
		Consider the orthogonally invariant set $\M = \{X\in \RR^{2\times 2} \ : \  \det(X) = 0\}$. We may 
		write $\M=\sigma^{-1}(S)$ where $S:=\{ (x_1,x_2)\in \RR^2  \ : \   x_1 x_2  =0\}$ is absolutely symmetric.
		Then ${\rm C}_S((1,1)) = \{ (1,0), (0,1) \}$. On the other hand, one can verify that 
		$uu^\top\in {\rm C}_{\M}(I_2)$ for each $u\in\R^2$ of norm one, where $I_2$ is the $2\times 2$ identity matrix.
	}
\end{example}

We now apply Theorem~\ref{thm:spectral_calculus} to our two running examples.

\begin{example}[Matrices of rank at most $r$]
{\rm Fix $r\leq n$ and recall the orthogonally invariant set 
$$
\RR^{n\times t}_r = \{X\in \RR^{n\times t} \ : \  \rank(X) \leq r\}.
$$
This set is a determinantal variety cut out by the $(r+1)\times (r+1)$ minors of a symbolic $n\times t$ matrix. 
One sees that 
$\sigma^{-1}(\RR^n_r) = \RR^{n\times t}_r$ 
where $\RR^n_r$ is the absolutely symmetric set defined in \eqref{eq:def Rnr}.
Theorem \ref{thm:spectral_calculus} and Example~\ref{ex:rank} together imply that 
$$
C^\#(\RR^{n\times t}_r) = C^\#(\RR^n_r) = {n \choose r}. 
$$
This number is also the ED degree of $\RR^{n\times t}_r$ (cf. \cite[Example 2.3]{DHOST}); we will revisit this point in Section~\ref{sect:eddegree}. In particular, the ED critical points of a general data matrix $Y$, with singular value decomposition 
$Y = U (\Diag \, \sigma(Y))V^\top$, are the ${n \choose r}$ matrices obtained by setting to zero all possible choices of 
$n-r$ singular values in this singular value decomposition. The matrix 
$$U \  (\Diag (\sigma_1(Y), \ldots, \sigma_r(Y), \underbrace{0, \ldots, 0}_{n-r})) \  V^\top $$ is a  nearest element of $\RR^{n\times t}_r$ to $Y$.
This is precisely the statement of the classical {\em Eckart-Young theorem}.
}
\end{example}

\begin{example}[Essential variety]  \label{ex:essential variety}
{\rm 
The essential variety $\E$ defined in \eqref{eqn:essential}
is  orthogonally invariant with 
$
\E = \sigma^{-1}(E_{3,2})
$
where $E_{3,2}$ is the absolutely symmetric set defined in Example \ref{ex:Enk}.
Then by Theorem \ref{thm:spectral_calculus} and Example~\ref{ex:Enk},
$$
C^\#(\E) = C^\#(E_{3,2}) = 6.
$$
Furthermore, for a matrix $Y\in \RR^{3\times 3}$ along with a singular value decomposition
$
Y = U (\Diag \sigma(Y)) V^\top,
$
by (\ref{eq:spectral proj}), the matrix 
$$
U \ \left(\Diag \lt( \f{\sigma_1(Y)+\sigma_2(Y)}{2}, \f{\sigma_1(Y)+\sigma_2(Y)}{2}, 0\rt) \right) \ V^\top
$$
is a nearest element of $\E$ to $Y$. 
This is precisely Hartley's result \cite[Theorem 5]{hartley} which is well known in the 
computer vision community. The six ED critical points of a general $Y = U (\Diag \, \sigma(Y))V^\top$ are the matrices 
$U(\Diag \, x)V^\top$ where $x$ varies over the following vectors:\\
\begin{small}
$$
\begin{array}{cc}
 \lt( \f{\sigma_1(Y)+\sigma_2(Y)}{2}, \f{\sigma_1(Y)+\sigma_2(Y)}{2}, 0\rt), & 
 \lt( \f{\sigma_1(Y)-\sigma_2(Y)}{2}, \f{-\sigma_1(Y)+\sigma_2(Y)}{2}, 0\rt), \\
  \lt( \f{\sigma_1(Y)+\sigma_3(Y)}{2}, 0, \f{\sigma_1(Y)+\sigma_3(Y)}{2}\rt), & 
 \lt( \f{\sigma_1(Y)-\sigma_3(Y)}{2}, 0, \f{-\sigma_1(Y)+\sigma_3(Y)}{2}\rt),\\
 \lt( 0,\f{ \sigma_2(Y)+\sigma_3(Y)}{2}, \f{\sigma_2(Y)+\sigma_3(Y)}{2}\rt), & 
 \lt(0, \f{\sigma_2(Y)-\sigma_3(Y)}{2}, \f{-\sigma_2(Y)+\sigma_3(Y)}{2}\rt). \\
\end{array}
$$
\end{small}
}
\end{example}

We end the section with the proof of the lemma evoked in the proof of Theorem~\ref{thm:spectral_calculus}.


\begin{proof}[Proof of Lemma~\ref{lem:meas}]
	Before we get to the main part of the proof, we need some notation. To this end, 	
	define $\mathcal{O}^{t}_n$ to be the set of $t\times n$ matrices with orthonormal columns. 
          Note that $\mathcal{O}^{t}_n$ is a smooth manifold of dimension $nt-\frac{n(n+1)}{2}$, and is called a {\em Stiefel manifold}; 
          see for example \cite[Proposition A.4]{burgisser}.
          In particular $\mathcal{O}^n$ is a manifold of dimension $\frac{n(n-1)}{2}$.
           We deduce that the product manifold
	$\mathcal{O}^n\times\mathcal{O}^{t}_n\times \R^n$ has dimension exactly $nt$, same as $\R^{n\times t}$.
	For any vector $x\in\R^n$ let $\textrm{sDiag } x$ denote the $n\times n$ diagonal matrix with $x$ on the diagonal.
	Define now the mapping $\Gamma\colon \mathcal{O}^n\times\mathcal{O}^{t}_n\times \R^n \to \R^{n\times t}$ by setting $\Gamma(U,V,x)=U(\textrm{sDiag } x) V^\top$. 
	Notice that for the absolutely symmetric set $S$, we have equality $\Gamma(\mathcal{O}^n\times\mathcal{O}^{t}_n\times S)=\sigma^{-1}(S)$.

	Suppose now that $S$ is a Lebesgue null set. Then clearly the set $\mathcal{O}^n\times \mathcal{O}^{t}_n\times S$ is Lebesgue null in the manifold $\mathcal{O}^n\times\mathcal{O}^{t}_n\times \R^n$. Since $\Gamma$ is smooth, the image $\Gamma(\mathcal{O}^n\times \mathcal{O}^{t}_n\times S)$, which coincides with $\sigma^{-1}(S)$, is Lebesgue null $\R^{n\times t}$, as claimed.
	Conversely, suppose that $\sigma^{-1}(S)$ is Lebesgue null.
	It is well-known that $\Gamma$ is a local diffeomorphism on an open full-measure subset
	of the domain manifold $\mathcal{O}^n\times \mathcal{O}^{t}_n\times \R^n$.
	To see this quickly, define the set 
	$$\Delta=\{x\in\R^{n}: |x_i|= |x_j| \textrm{ for some indices }i\neq j\}.$$
	Observe that since $\Gamma$ is a semi-algebraic map, we may stratify $\mathcal{O}^n\times \mathcal{O}^{t}_n\times \Delta^c$ into finitely many smooth manifolds ${\mathcal{M}_i}$ so that the restriction of $\Gamma$ to each $\mathcal{M}_i$ has constant rank. Suppose now for the sake of contradiction that the derivative of $\Gamma$ has deficient rank on some maximal dimensional manifold $\mathcal{M}_i$. Then by the constant rank theorem \cite[Theorem 5.13]{Lee2},
	there exists a nontrivial smooth path $(U(t),V(t),x(t))$ in $\mathcal{M}_i$ 
	so that	$\Gamma$ is constant on the path, meaning that the matrices $U(t)(\textrm{sDiag } x(t)) V(t)^\top$ are equal for all $t$. On the other hand, taking into account that the coordinates of $x(t)$ are distinct and appealing to uniqueness of singular values and singular vectors, we obtain a contradiction. 
Hence $\Gamma$ is a local diffeomorphism on an open full measure subset of $\mathcal{O}^n\times \mathcal{O}^{t}_n\times \R^n$. 
It follows immediately that $S$ is Lebesgue null, since otherwise the image $\Gamma(\mathcal{O}^n\times \mathcal{O}^{t}_n\times S)$ would fail to be Lebesgue null. This completes the proof.
\end{proof}

\section{Applications} \label{sect:applications}

We now apply the techniques developed in the last section to calculate and count the ED critical points of several orthogonally invariant matrix sets.

\begin{example}[Matrices orthogonally equivalent to a given matrix]   \label{ex:OE}
{\rm
Fix a matrix $A\in \RR^{n\times t}$. We say $X\in \RR^{n\times t}$ is 
{\em orthogonally equivalent} to $A$ if $X = UAV^\top$ for some $U\in \mathcal{O}^n$ and $V\in \mathcal{O}^t$. 
Let $\mathcal{M}_A\subseteq \RR^{n\times t}$ be the set of matrices that are orthogonally equivalent to $A$. 
Then $\M_A$ is orthogonally invariant and $\M_A= \sigma^{-1} (\sigma(A))$. By  
Theorem \ref{thm:spectral_calculus}, 
$$C^\#(\mathcal{M}_A) = |\Pi_n^\pm \{\sigma(A)\}| = 2^n \f{n!}{n_1! \ldots n_k!}$$
where $\sigma(A) = (\underbrace{\sigma_1,\ldots,\sigma_1}_{n_1}, \ldots,\underbrace{\sigma_k,\ldots,\sigma_k}_{n_k} )$.
}
\end{example}

\begin{example}[Real orthogonal group] \label{ex:On}
{\rm
The real orthogonal group $\mathcal{O}^n$ is a special case of the previous example 
because $\mathcal{O}^n$ is the set of $n\times n$ real matrices that are orthogonally equivalent to
the $n\times n$ identity matrix, namely, 
$$\mathcal{O}^n = \sigma^{-1} ((\underbrace{1,\ldots,1}_{n})).$$
Hence $C^\#(\mathcal{O}^n) = 2^n$ which is also the ED degree of $\mathcal{O}^n$
(cf. \cite[Theorem 3.2]{draisma}).

By Proposition~\ref{prop:proj transfer}, for any $Y\in \RR^{n\times n}$ with singular value decomposition
$Y  = U (\Diag \sigma(Y)) V^\top$ where  $U,V\in \mathcal{O}^n$,  
one has $UV^\top \in {\rm proj}_{\mathcal{O}^n} (Y)$. Moreover, if the singular values of $Y$ are distinct, then by 
Theorem~\ref{thm:spectral_calculus}, the set of 
critical points of $Y$ on $\mathcal{O}^n$ is $\{ U ( {\rm Diag } \,x)V^\top \,:\, x_i = \pm 1 \,\,\forall \,\,i \}$.
These results were also obtained by Draisma and Baaijens \cite{draisma} using algebraic methods.
}
\end{example}

\begin{example}[Unit sphere of Schatten $d$-norm]  \label{ex:unit ball}
{\rm
The Schatten $d$-norm of a matrix $X\in \RR^{n\times t}$ is defined as 
$$
\|X\|_d := \lt[  \sum_{i=1}^n (\sigma_i (X))^d  \rt]^{1/d}.
$$
Note the following fact: $\|U X V^\top\|_d = \|X\|_d$ for any $X\in \RR^{n\times t}$, $U\in \mathcal{O}^n$, 
$V \in \mathcal{O}^t$. 
Hence the unit sphere of the Schatten $d$-norm
$$
\mathcal{F}_{n,t,d} := \{X\in \RR^{n\times t} \ : \   \|X\|_d = 1\} $$ 
is orthogonally invariant. Assume $d\geq 2$ is an even integer. One has 
$
\mathcal{F}_{n,t,d} = \sigma^{-1}(F_{n,d})
$
where $F_{n,d}$ is the {\em affine Fermat hypersurface} 
\begin{align*} 
F_{n,d} := \lt\{x\in \RR^n \ : \  \sum_{i=1}^n x_i^d = 1\rt\}
\end{align*}
which is absolutely symmetric. Therefore, $C^\#(\mathcal{F}_{n,t,d}) = C^\#(F_{n,d})$ for all $t$.

\begin{figure}
\centering
\includegraphics[scale = 0.35]{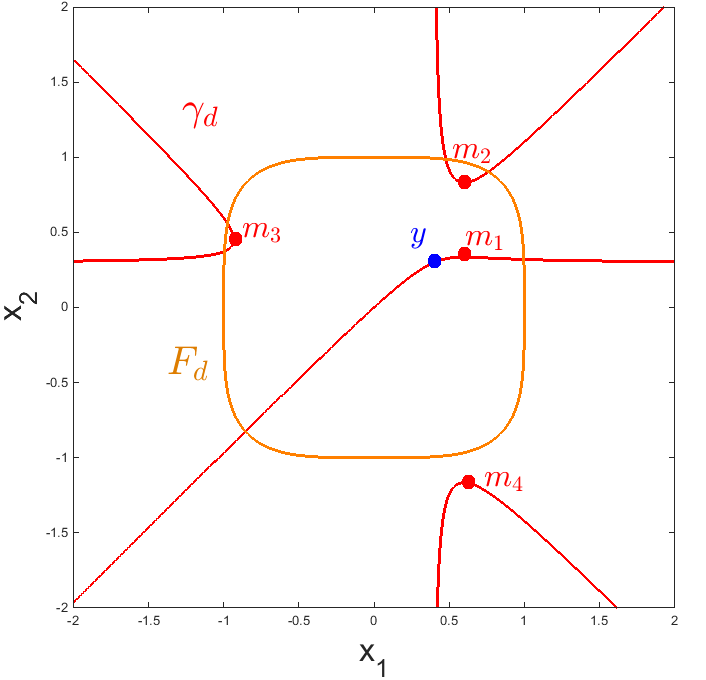}
\caption{The curves $F_d$ and $\gamma_d$ in Example \ref{ex:unit ball} for $d=4$} \label{fig:fermat}
\end{figure}

Consider the special case $n=2$.  For simplicity, we set $F_d := F_{2,d}$. Then $F_2$ is the unit circle 
and $C^\#(F_{2}) = 2$. Suppose $d\geq 4$.
For any $y\in \RR^2$, a point $x$ is an ED critical point of $y$ on $F_d$ if and only if 
$x$ lies on both $F_d$ and the curve 
\begin{align} \label{eq:fermat system}
 \gamma_d:= \{x\in \RR^2 \ : \  x_1^{d-1}(x_2-y_2) = x_2^{d-1}(x_1-y_1)\}.
\end{align}
The graph of $\gamma_4$ is shown in Figure \ref{fig:fermat}.
By symmetry, we may assume that $y_1>y_2>0$. 
When $y_1$ is small, the ``optimal'' points, $m_1,m_2,m_3,m_4$, 
on the pieces of $\gamma_d$ in the coordinate directions lie inside the curve $F_d$, namely,  
$$m_i \in \{x\in \RR^2 \ : \  x_1^d + x_2^d <1\}$$ for $i =1,\ldots,4$.
Hence, 
$C^\#_{F_{d}}(y) \leq 8$ for any point $y$ except on a null set, and 
there is an open set $V$ in $\RR^2$ such that $C^\#_{F_{d}}(y) = 8$ for any $y\in V$.
To sum up, by Theorem \ref{thm:spectral_calculus}, we have
\begin{align*} 
C^\# (\mathcal{F}_{2,t,d}) = C^\#(F_d) = 
\begin{cases}
2 & \mbox{ if } d =2\\
8 & \mbox{ if } d\geq 4.
\end{cases}
\end{align*}
}
\end{example}

\begin{example}[$SL_n^\pm$]  \label{ex:special linear}
{\rm
This example was considered in \cite[Section 4]{draisma}. Define the orthogonally invariant set 
$$
SL_n^\pm := \{X\in \RR^{n\times n} \ : \   \det (X) = \pm 1\} = \sigma^{-1}(H_n)
$$
where  $H_n$ is the absolutely symmetric set defined as 
\begin{align*} 
H_n := \{x\in \RR^n \ : \  x_1 \cdots x_n = \pm 1\}.
\end{align*}
Consider the special case $n=2$. 
We evaluate $C^\# (SL^\pm_2)$ by computing $C^\#(H_2)$.
The set $H_2$ is a disjoint union of two hyperbolas:
$$
H_2^\pm := \{ x\in \RR^2 \ : \  x_1x_2 = \pm 1\}.
$$
Since these hyperbolas are disjoint, for any $y$, one has 
$$
C^\#_{H_2}(y) = C^\#_{H_2^+}(y) + C^\#_{H_2^-}(y).
$$
Notice that each of the sets ${\rm C}_{H_2^\pm}(y)$ is defined by the roots of  
a univariate quartic. Indeed, 
$$ C_{H_2^+}(y) =  \left\{ \left(x,\frac{1}{x} \right) \in \RR^2 \ : \ q_y^+(x) := x^4- x^3 y_1 + xy_2-1 = 0 \right\}$$ 
and 
$$ C_{H_2^-}(y) = \left\{ \left(x, \frac{-1}{x} \right) \in \RR^2 \ : \ q_y^-(x) := x^4- x^3 y_1 - xy_2-1 = 0 \right\}.$$ 
By computing the trace forms of these quartics and applying Sylvester's criterion (see e.g. \cite[Corollary 2.9]{sturmfels2002}) 
we acquire these results:
$$
C^\#_{H_2}(y) = 
\begin{cases}
6 & \mbox{ if } D^+(y)>0 \text{ or } D^-(y)>0\\
4 & \mbox{ if } D^+(y)<0 \textup{ and } D^-(y)<0 
\end{cases}
$$
where the bivariate polynomials 
$$D^+(y) := -256 + 192y_1 y_2  + 6 y_1^2 y_2^2  
+ 4y_1^3 y_2^3   -27 y_1^4   - 27 y_2^4$$
and  
$$D^-(y) := -256 - 192y_1 y_2  + 6 y_1^2 y_2^2  
- 4y_1^3 y_2^3   -27 y_1^4   - 27 y_2^4$$
are the discriminants of $q^+_y$ and $q^-_y$ respectively; see Figure \ref{fig:chamber}.
Then by Theorem \ref{thm:spectral_calculus}, we know 
$$C^\# (SL^\pm_2) = 6.$$ 
}
\end{example}

\begin{figure}
\includegraphics[scale = 0.5]{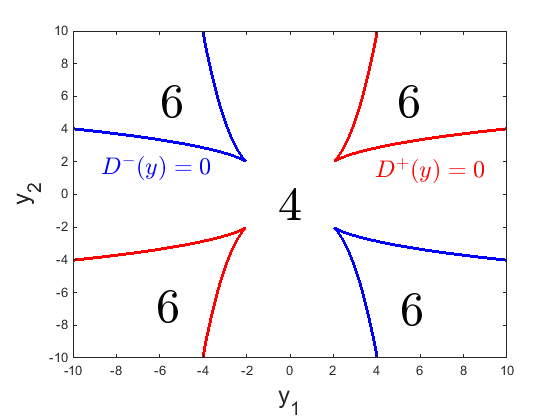}
\caption{Values of $C^\#_{H_2}(y)$ for general $y$, in Example \ref{ex:special linear}} \label{fig:chamber}
\end{figure}

\section{Connections to the Euclidean distance degree} \label{sect:eddegree}
As mentioned in the introduction, this work was inspired by the general framework in \cite{DHOST} for counting the critical points of the squared Euclidean distance function on an algebraic variety. In this section we comment on the relationships between our results and those in \cite{DHOST}, beginning 
with some basic facts about real varieties \cite{whitney,BCR}.

Consider a real variety  $\V \subseteq \RR^n$ whose vanishing ideal 
is generated by the polynomials $f_1,\ldots,f_s \in \RR[x_1,\ldots,x_n]$. 
Then the {\em Zariski closure} of $\V$ is $\V_\CC := \{ x \in \CC^n \,:\, f_1(x) = \ldots = f_s(x)=0 \}$, the smallest complex variety containing $\V$. The real part of $\V_\CC$, 
i.e., $\V_\CC \cap \RR^n$, is precisely $\V$ and hence, the vanishing ideal of 
$\V_\CC$ in $\CC[x_1,\ldots,x_n]$ is also generated by $f_1, \ldots, f_s$. Also, both $\V$ and $\V_\CC$ have the same dimension over $\R$ and $\C$ respectively. 
Recall that if $\V_\CC = \bigcup_{j=1}^t \mathcal{W}^j$ is a minimal irreducible decomposition of $\V_\CC$, 
then for any $j$, the real part of $\mathcal{W}^j$, denoted by $\mathcal{U}^j$, is an irreducible 
real variety whose Zariski closure is $\mathcal{W}^j$. Moreover, $\V = \bigcup_{j=1}^t \mathcal{U}^j$ 
is a minimal irreducible decomposition of $\V$.  For simplicity, in the rest of this section we will assume that $\V_\CC$ is irreducible.

Denote by $J(f)$ the $s \times n$ {\em Jacobian} matrix whose $(i,j)$-entry is $\frac{\partial f_i}{\partial x_j}$. Then a point $x \in \V_\CC$ is said to be {\em regular} 
if the rank of the Jacobian matrix evaluated at $x$, $\rank(J(f)(x))$, equals the codimension of $\V_\CC$. Let $\V_\CC^{\textup{reg}}$ denote the regular points of $\V_\CC$. It is known that $\V_\CC^{\textup{reg}}$ is a Zariski open subset of $\V_\CC$, and its complement is a proper subvariety in $\V_\CC$, denoted as ${\rm Sing}(\V_\CC)$, and called the {\em singular locus} of $\V_\CC$. 
The set $\V^{\rm reg}$ of regular points in $\V$ is precisely the set of real points in $\V_\CC^{\textup{reg}}$. However,
 while $\V_\CC^{\textup{reg}}$ is dense in $\V_\CC$ in the Euclidean topology, $\V^{\textup{reg}}$ may not be dense in $\V$ or even
 $\V^*$, the set of smooth (i.e., $C^p$-smooth) points in $\V$. The following example illustrates this behavior.
 
 \begin{example} [Cartan umbrella] \label{ex:cartan}
{\rm The {\em Cartan umbrella} is the real variety 
$$ \V = \{ (x_1,x_2,x_3) \in \RR^3 \,:\, x_3(x_1^2+x_2^2)-x_1^3=0 \}.$$
Here $\V_\CC$ is irreducible, as is $\V$. On the other hand, $\V$ is also the union of a surface (which is not a real variety) and the $x_3$-axis (Figure~\ref{fig:cartanumbrella}).
\begin{figure} 
    \includegraphics[scale=0.5]{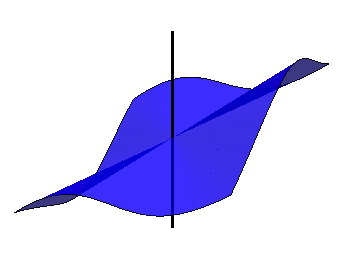}
    \caption{The Cartan umbrella}
    \label{fig:cartanumbrella}
\end{figure}
No point on the $x_3$-axis is regular while all points on the $x_3$-axis, except the origin, are smooth.}
\end{example}

This example shows an important distinction between smooth and regular points on a real variety. Our goal in this section is 
to elaborate on this difference in the context of critical points as defined in \cite{DHOST}. 
 Recall that the normal space at a regular point $x \in \V_\CC$, denoted as  $\mathcal{N}^{\rm alg}_{\V_\CC}(x)$, is the row space 
of $J(f)(x)$.
Given a data point $y \in \CC^n$, and a real variety $\V$, the authors of \cite{DHOST} study 
the set of critical points 
 \begin{align} \label{eq:alg crit point}
C^{\rm reg}_{\V_\CC}(y) := \left\{ x \in \V_\CC^{\rm reg} \,:\, y-x \in \mathcal{N}^{\rm alg}_{\V_\CC}(x) \right\}.
\end{align}
Note that in the above algebraic setting, even though our input is a real variety $\V \subseteq \RR^n$, 
there is a natural passage into the complex numbers; the real variety $\V$ is replaced by the 
complex variety $\V_\CC$, the data point can be any $y \in \CC^n$, and the critical points of 
$y$ are regular points in $\V_\CC$.
This complexification yields the first key result in \cite{DHOST}, namely that for a fixed $\V$, the number of 
critical points of a general data point $y \in \CC^n$ is a constant. In other words, the cardinality of $C^{\rm reg}_{\V_\CC}(y)$, 
which we will denote as $C^{\rm \# \,reg}_{\V_\CC}(y)$, is a constant. 
This constant is called the {\em Euclidean distance degree (EDdegree)} of $\V$ in \cite{DHOST}. The phrase ``general data point" refers 
to the fact that $C^{\rm \# \,reg}_{\V_\CC}(y)$ is different from ${\rm EDdegree}(\V)$ only if $y$ lies on 
certain proper subvarieties in $\CC^n$ which 
we will describe in the Appendix. We remark that if $\V_\CC$ is reducible, then ${\rm EDdegree}(\V)$ is the sum of the ED degrees of its 
irreducible components.

Our results in this paper provide a method to compute the (smooth) ED critical points of a real orthogonally invariant set of matrices. We did not require this 
set to be a real variety, and the notion of smoothness (of critical points) was geometric 
and not algebraic. Therefore, in order to connect our work to that in \cite{DHOST}, we need to restrict to real varieties 
$\V \subseteq \RR^n$,  data points $y \in \RR^n$, 
and understand the relationships between the following two sets of critical points of $y$:
\begin{itemize}
\item $C_{\V}(y)$, the set of ED critical points of $y$ on $\V$ as in Section~\ref{sect:ED critical points}, 
and 
\item $C_{\V}^{\rm reg}(y)$, the set of regular critical points of $y$ on $\V$.
\end{itemize}
As usual we denote the cardinalities of these sets by $C^\#_\V(y)$ and $C_\V^{\rm \# \,reg}(y)$ respectively.
Also, note that $C_{\V}^{\rm reg}(y) = C^{\rm reg}_{\V_\CC}(y) \cap \RR^n = C^{\rm reg}_{\V_\CC}(y) \cap \V$.
Since a regular point of $\V$ is $C^p$-smooth for $p\in \{\infty,\omega\}$,  for any $y$ we have,
\begin{align} \label{eq:reg is contained in smooth}
C_{\V}^{\rm reg}(y) \subseteq C_{\V}(y)\,\,\,\, \textup{ and so,} \,\,\,\, C_\V^{\rm \# \,reg}(y) \leq C^\#_\V(y).
\end{align}
However, for a given $y$, this inequality can be strict since there can be smooth points on $\V$ that are not regular as we saw in the Cartan umbrella. 
In fact, in the Cartan umbrella there is a dense set of $y$ for which the inequality in \eqref{eq:reg is contained in smooth} is strict since for any $y$ with $y_3 \neq 0$, there will be an ED critical point on the $x_3$-axis. 

Moving on to invariants of the entire set $\V$, on the one hand, we have the quantity
$C^\#(\V) = \inf_{\Gamma\in \Theta } \sup_{y\in \Gamma^c} C^\#_\V(y)$ from Section~\ref{sect:ED critical points}, 
and on the other hand, we have the constant ${\rm EDdegree}(\V)$ which is an upper bound to 
$C_{\V_\CC}^{\rm \# \, reg}(y)$ whenever $C_{\V_\CC}^{\rm reg}(y)$ is finite.
Since the inequality in \eqref{eq:reg is contained in smooth} can be strict, it is not clear that the number $C^\#(\V)$ is 
always a lower bound to $\textup{EDdegree}(\V)$.  However, it is true under certain conditions.

\begin{theorem} \label{thm:cV at most EDdegree}
If $\V$ is a real variety such that $\V^{\rm reg}$ is Euclidean dense in $\V^*$, then we have the inequality 
$C^\#(\V)\leq {\rm EDdegree}(\V)$.
\end{theorem}

This is an immediate consequence of the following lemma. Note that the Cartan umbrella does not satisfy the 
assumptions of Theorem~\ref{thm:cV at most EDdegree}.

\begin{lemma}  \label{lem:suff cond}
Consider a real variety $\V \subseteq \RR^n$ (possibly reducible) such that $\V^{\rm reg}$ is Euclidean dense in $\V^*$.
Then there is an open set of points $y\in \RR^n$ such that 
\begin{align} \label{lem:condition3}
& \text{for any irreducible component $\mathcal{W}$ of $\V$,  }
\text{$ C_\mathcal{W_\CC}^{\rm  \# \,reg} (y) = {\rm EDdegree}(\mathcal{W})$,}\\
& C^\#_\V(y) = C^\#(\V), \ \text{ and }     \label{lem:condition1} \\
& C_{\V}(y)\subseteq C_{\V}^{\rm reg}(y).\label{lem:condition2}
\end{align}
\end{lemma}

\begin{proof}
Fix an irreducible component $\mathcal{W}$ of $\V$. Then from \cite[Theorem 4.1]{DHOST} the set of all $y \in \CC^n$ for which 
$C^{\rm \# \, reg}_{\mathcal{W}_\CC}(y)\neq {\rm EDdegree}(\mathcal{W})$  lies in a proper subvariety 
$D$ in  $\CC^n$. 
Consider the real variety $D_\RR:= D\cap \RR^n$ which contains all the real points $y$ with 
$C^{\rm \# \, reg}_{\mathcal{W}_\CC}(y)\neq {\rm EDdegree}(\mathcal{W})$.
The Zariski closure of $D_\RR$ is a subvariety of $D$ with complex dimension equal to  the real dimension of $D_\RR$.
Therefore, the real dimension of $D_\RR$ is smaller than $n$, and hence $D_\RR$ is a proper subvariety in 
$\RR^n$. In particular, the complement of $D_\RR$ in $\RR^n$ is an open full measure set. Taking the intersection of all such 
open full measure sets as we vary over the irreducible components of $\V$, we still have an open full measure set, and any point 
$y$ in this set satisfies \eqref{lem:condition3}.

Since $\VV$ is a variety, one has $C^\#(\VV)<\infty$ by Proposition \ref{prop:semialgebraic}. 
Then by quantifier elimination, the set 
$\{y\in \RR^n \ : \   C^\#_{\VV}(y) = C^\#(\V)\}$ is semialgebraic. Moreover, this set cannot be Lebesgue null by the definition of $C^\#(\V)$. Therefore, it must contain an open subset. 
Hence \eqref{lem:condition1} 
holds for all $y$ in some open set.

We now show that \eqref{lem:condition2} holds for any point $y$ in an open full measure set, which will prove the lemma.
For any $i= 0, 1,\ldots, n$ consider the submanifold
$$
M_i := \{x\in \VV^* \ : \  {\rm dim}(\mathcal{T}_{\V}(x))=i \}.
$$
These submanifolds $M_i$ clearly partition $\VV^*$. We first claim $M_i\cap {\rm cl}(M_j)=\emptyset$ for all distinct pairs $i,j \in [n]$. 
Note the symbol ``cl" here denotes the Euclidean closure.
Indeed, if there existed some point $x\in M_i \cap {\rm cl}(M_j)$, then we would deduce for all points $y\in M_j$ sufficiently near $x$, the equality $j=\dim \mathcal{T}_{\mathcal{V}}(y)= \dim \mathcal{T}_{\mathcal{V}}(x) =  i$, a contradiction. The middle equality follows from the fact that $\dim \mathcal{T}_\V( \cdot )$ is constant in a neighborhood of $x$ in $\V$.

Now since $\VV^{\rm reg}$ is dense in $\VV^*$, the inclusion $\V^* \subseteq {\rm cl}(\V^{\rm reg}) =  \bigcup_i {\rm cl}(\V^{\rm reg} \cap M_i)$ holds.
We claim that for any $i$, the set $\VV^{\rm reg} \cap M_i$ is Euclidean dense in $M_i$, 
namely, $M_i \subseteq {\rm cl}(\VV^{\rm reg} \cap M_i)$. 
To see this, it suffices to establish $M_i \cap {\rm cl}(\VV^{\rm reg}\cap M_j)=\emptyset$ for any $j\neq i$. This follows immediately by observing
$M_i \cap {\rm cl}(\VV^{\rm reg}\cap M_j)\subseteq M_i \cap {\rm cl}(M_j) =\emptyset$. 
It follows that for any $i$, the strict inequality 
$$
\dim (M_i \setminus (\VV^{\rm reg}\cap M_i)) < \dim M_i
$$
holds (see \cite[Proposition 3.16]{costeIntrotoSAG}).
By dimension arguments, the dimension of the set 
$$
\bigcup_{x\in M_i \setminus (M_i \cap \VV^{\rm reg})} (x+\N_\VV(x)) 
$$
is strictly less than $n$. In particular, its interior is empty. 
Taking the union over $i$, 
$$
\bigcup_{x\in \VV^* \setminus \VV^{\rm reg}}(x+\N_\VV(x)) 
$$
also has empty interior. 
Any $y$ which is not in this set satisfies \eqref{lem:condition2}.
\end{proof}

We conclude the paper by comparing $C^\#(\V)$ and ${\rm EDdegree}(\V)$ in some of the examples 
we saw in Section~\ref{sect:real critical points of spectral sets} and Section~\ref{sect:applications}.

\begin{example} [$SL_n^{\pm}$] \label{ex:lb sln}
{\rm
Every point in the  Zariski closure of $SL_n^\pm$ (c.f. Example \ref{ex:special linear})
is regular. Therefore by Theorem \ref{thm:cV at most EDdegree}, we obtain
$$
C^\#(SL_n^\pm) \leq {\rm EDdegree}(SL_n^\pm) = n2^n.
$$
The formula for EDdegree was derived in \cite{draisma}.
We saw in Example \ref{ex:special linear} that  strict inequality holds when $n=2$.
}
\end{example}

\begin{example} [Unit sphere of Schatten $d$-norm] \label{ex:lb fermat}
{\rm
Consider the set $\mathcal{F}_{2,2,d} \subseteq \RR^{2\times 2}$ given in Example \ref{ex:unit ball}.
For any $d$  the Zariski closure of   $\mathcal{F}_{2,2,d} $ is regular everywhere.
By Theorem \ref{thm:cV at most EDdegree}, 
$$
 {\rm EDdegree}(\mathcal{F}_{2,2,d}) \geq C^\#(\mathcal{F}_{2,2,d}) = 8,
$$
for $d\geq 4$.
From {\tt Macaulay2} \cite{M2} computations,  ${\rm EDdegree}(\mathcal{F}_{2,2,d})$ equals $16,34,64,98$ when $d=4,6,8,10$ respectively. This suggests the gap between $C^\#(\mathcal{F}_{2,2,d} )$ and  ${\rm EDdegree}(\mathcal{F}_{2,2,d})$
increases with $d$.
}
\end{example}


\begin{example}
{\rm 
Let $\mathcal{E}$ be the  essential variety from \eqref{eqn:essential}. 
It is an irreducible variety of codimension three in $\RR^{3\times 3}$
\cite[Proposition 3.6]{demazure1988}. 
By the transfer of smoothness (see the paragraph after the proof of Proposition~\ref{prop:proj transfer}), 
we know $\mathcal{E}^\ast = \mathcal{E}\setminus\{0\}$. 
Moreover, a straightforward computation verifies that any non-zero essential matrix is a regular point of $\E$.
Thus $\mathcal{E}^{\rm reg} = \mathcal{E}^\ast = \mathcal{E}\setminus\{0\}$, and by Theorem~ \ref{thm:cV at most EDdegree}, 
we have 
\begin{equation} \label{eq:eddegree e variety}
{\rm EDdegree}(\mathcal{E}) \geq  C^\#(\mathcal{E})=6.
\end{equation}
With completely different tools, but entirely motivated by this work, one can prove the equality
${\rm EDdegree}(\mathcal{E})=6$. The details will be described in a forthcoming paper.
}
\end{example}



\appendix\section{}

The Cartan umbrella illustrates a yet unexplored feature of the 
EDdegree concerning the exceptional loci of points $y$ at which $C^{\rm \# \, reg}_{\V_\CC}(y)$ can be different from ${\rm EDdegree}(\V)$.
We comment briefly on this topic which deserves further investigation.

For each $y \in \CC^n$,  the set $C^{\rm  reg}_{\V_\CC}(y)$  is the variety of a polynomial ideal called the {\em critical ideal} of $y$ (see (2.1) in \cite{DHOST}), and typically, this variety has $\textrm{EDdegree}(\V)$-many distinct complex solutions. 
There are two ways in which the critical ideal of $y$ can have fewer distinct roots; the first is because of roots with multiplicity, and the second is because a root may wander off into ${\rm Sing}(\V_\CC)$ due to closure issues. These situations create two exceptional loci in the space of data points $y$. The locus of $y \in \CC^n$ for which the critical ideal has roots with multiplicity  is called the {\em ED discriminant} of $\V_\CC$, and is denoted as $\Sigma_{\V_\CC}$.  The ED discriminant is typically a hypersurface in $\CC^n$ and can be computed from the equations of $\V$ (see Section 7 in \cite{DHOST} for algorithms and examples). We saw the ED discriminant of the parabola $\{(x_1,x_2) \in \RR^2 \,:\, x_2=x_1^2 \}$ in Figure~\ref{fig:parabola}. 
It is the curve defined by 
$$16y_2^3-27y_1^2-24y_2^2+12y_2-2 = 0.$$

The second type of exceptional locus has not been studied in \cite{DHOST} and was suggested to us by Bernd Sturmfels. To describe it, let us denote the copy of $\CC^n$ that contains $\V_\CC$ as $\CC^n_x$ and the copy that contains the data points $y$ as $\CC^n_y$. The {\em ED correspondence} of $\V_\CC$, denoted as $\mathbb{E}_{\V_\CC}$ and described in \cite[Section 4]{DHOST}, is the Zariski closure of:
\begin{align}
	\left\{ (x,y) \in \CC^n_x \times \CC^n_y \,:\, x \in \V_\CC^{\rm reg}, \,\,y \in \CC^n, \,\,x-y \in \N_{\V_\CC}^{\rm alg} (x) \right\}.
\end{align}
By \cite[Theorem~4.1]{DHOST}, the ED correspondence is an irreducible variety of dimension $n$ in $\CC^n_x \times \CC^n_y$. It admits two natural projections $\pi_x$ and $\pi_y$ into $\CC^n_x$ and $\CC^n_y$ respectively. Over general $y \in \CC^n$, the projection $\pi_y \,:\, \mathbb{E}_{\V_\CC} \rightarrow \CC^n_y$ has finite fibers of cardinality equal to ${\rm EDdegree}(\V)$. 
The exceptional locus we are looking for is the set of data points $y \in \CC^n$ that have critical points that fall into the 
singular locus ${\rm Sing}(\V_\CC)$. 
This is precisely the Zariski closure of the set 
\begin{align} \pi_y \left( \mathbb{E}_{\V_\CC} \cap ({\rm Sing}(\V_\CC) \times \CC^n_y)    \right)
\end{align}
which we call the {\em ED data singular locus} of $\V_\CC$, and denote by $D_{\V_\CC}$.  
This affine scheme deserves further study, and we illustrate 
both $D_{\V_\CC}$ and $\Sigma_{\V_\CC}$ for the Cartan umbrella in Example~\ref{ex:cartan_contd}. Note that $D_{\V_\CC}$ is empty for the parabola in Figure~\ref{fig:parabola} since its singular locus is empty.
The upshot of the above discussion is that for all $y \in \RR^n \setminus D_{\V_\CC}$ but in a connected component of $\RR^n_y \setminus \Sigma_{\V_\CC}$, the number of regular critical points on $\V$, $C^{\rm \# \, reg}_{\V}(y)$,  is a constant. 

\begin{example} [Cartan umbrella continued] \label{ex:cartan_contd}
	{\rm 
		The ED discriminant of the Cartan umbrella in Example~\ref{ex:cartan} is a degree 12 surface in the space of 
		data points given by the following equation:\\
		\begin{small}
			$$
			\begin{array}{c}
			256y_1^{12}-35328y_1^{10}y_2^2-108984y_1^8y_2^4-111867y_1^6y_2^6-93975y_1^4y_2^8-9216y_1^2y_2^{10}\\
			-2048y_2^{12}-2304y_1^{11}y_3-2112y_1^9y_2^2y_3-149280y_1^7y_2^4y_3-116868y_1^5y_2^6y_3\\
			+53532y_1^3y_2^8y_3+34560y_1y_2^{10}y_3+6912y_1^{10}y_3^2+14016y_1^8y_2^2y_3^2-28764y_1^6y_2^4y_3^2\\
			+41502y_1^4y_2^6y_3^2-86430y_1^2y_2^8y_3^2-768y_2^{10}y_3^2-7936y_1^9y_3^3+150720y_1^7y_2^2y_3^3\\
			-200148y_1^5y_2^4y_3^3-411728y_1^3y_2^6y_3^3+1476y_1y_2^8y_3^3+9216y_1^8y_3^4-46656y_1^6y_2^2y_3^4\\
			+31908y_1^4y_2^4y_3^4+110817y_1^2y_2^6y_3^4+4953y_2^8y_3^4-27648y_1^7y_3^5+23808y_1^5y_2^2y_3^5\\
			+91236y_1^3y_2^4y_3^5-40284y_1y_2^6y_3^5+28672y_1^6y_3^6-196992y_1^4y_2^2y_3^6-240480y_1^2y_2^4y_3^6\\
			-2592y_2^6y_3^6-9216y_1^5y_3^7+14208y_1^3y_2^2y_3^7+28800y_1y_2^4y_3^7+27648y_1^4y_3^8\\
			+39168y_1^2y_2^2y_3^8+2304y_2^4y_3^8-27648y_1^3y_3^9-27648y_1y_2^2y_3^9=0.
			\end{array}
			$$
		\end{small}
		
		The EDdegree of the Cartan umbrella is seven. The ED discriminant partitions $\RR^3$ into two regions where the real regular 
		critical points are either one or three, while the number of ED critical points is two or four. In particular, for all $y \in \RR^3$ 
		for which the polynomial defining the ED discriminant is positive, we get three real regular critical points.
		
		\begin{figure}
			\includegraphics[scale=0.4]{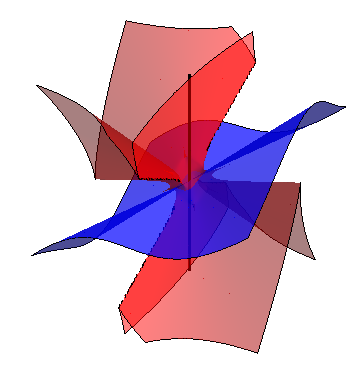}
			\includegraphics[scale=0.4]{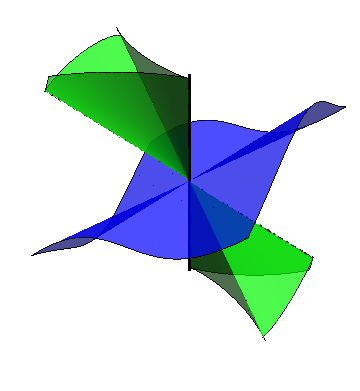}
			\caption{The Cartan umbrella with its ED discriminant on the left and ED data singular locus on the right.}
			\label{fig:cartanumbrellaWloci}
		\end{figure}

		The ED data singular locus, $D_{\V_\CC}$,  is the reducible surface in $\CC^3$ defined by
		$$
		(y_1^2+y_2^2) \cdot (4y_1^4+8y_1^2y_2^2+4y_2^4+4y_1^3y_3+36y_1y_2^2y_3+27y_2^2y_3^2) = 0.
		$$
		(see Figure  \ref{fig:cartanumbrellaWloci})
		The data point $y=(-2,-1,2)$ lies on $D_{\V_\CC}$ but not on $\Sigma_{\V_\CC}$. The critical ideal of this point 
		has seven distinct complex roots but one of them is singular. Among the remaining six roots, two are real and, 
		$(-2,-1,2)$ lies in the region where the polynomial defining $\Sigma_{\V_\CC}$ is positive. 
		
	}
\end{example}

{\bf Acknowledgments.} We thank Jan Draisma, Giorgio Ottaviani, and Bernd Sturmfels for helpful discussions. 
In particular, the notion of ED data singular locus was suggested to us 
by Sturmfels.

\bibliographystyle{plain}
\bibliography{bibliography}


\end{document}